\documentclass{amsart}
\usepackage{amssymb
,amsthm
,amsmath,dirtytalk
,amscd,mathrsfs
,mathtools,enumerate
}

\usepackage{hyperref}
\usepackage[all]{xy}
\usepackage[top=30truemm,bottom=30truemm,left=25truemm,right=25truemm]{geometry}

\newcommand{\SL}{\mathrm{SL}}
\newcommand{\GL}{\mathrm{GL}}

\newtheorem{thm}{Theorem}[section]

\newtheorem{lem}[thm]{Lemma}
\newtheorem{prop}[thm]{Proposition}
\newtheorem{coro}[thm]{Corollary}

\theoremstyle{remark}
\newtheorem{rem}[thm]{Remark}

\theoremstyle{definition}
\newtheorem{defn}[thm]{Definition}

\numberwithin{equation}{section}

\makeatletter
\def\iddots{\mathinner{\mkern1mu\raise\p@
	\hbox{.}\mkern2mu\raise4\p@\hbox{.}\mkern2mu
	\raise7\p@\vbox{\kern7\p@\hbox{.}}\mkern1mu}}
\def\adots{\mathinner{\mkern2mu\raise\p@\hbox{.}
 \mkern2mu\raise4\p@\hbox{.}\mkern1mu
 \raise7\p@\vbox{\kern7\p@\hbox{.}}\mkern1mu}}
\makeatother

\allowdisplaybreaks
\title[Mirabolic subgroup]{The action of a mirabolic subgroup on a symmetric variety}
\author{Hengfei LU}
\address{Faculty of Mathematics, University of Vienna, Oskar-Morgenstern-Platz 1, Wien 1090, Austria}
\email{hengfei.lu@univie.ac.at}
\begin{document}

\begin{abstract} Let $F$ be a  local field of characteristic zero. Let $E$ be a quadratic field extension of $F$.
 We show that any $P_E$-invariant 
linear functional on a $\GL_n(E)$-distinguished irreducible smooth admissible representation of $\GL_{2n}(F)$ is also $\GL_n(E)$-invariant, where $P_E$ is the standard mirabolic subgroup of $\GL_{n}(E)$.
\end{abstract}
\maketitle
\tableofcontents

\section{Introduction}
Let $F$ be a local field of characteristic zero. Let $E=F[\delta]$ be a quadratic field extension of $F$ with $\delta^2\in F^\times\setminus(F^\times)^2$.
Let $Mat_{n,n}(F)$ (resp. $Mat_{n,n}(E)$) denote the set of all $n\times n$ matrices over $F$ (resp. $E$). 
Let $\GL_n(F)$ act on $Mat_{n,n}(F)$
by inner conjugation. Let $P_F$ be the mirabolic subgroup of $\GL_n(F)$ consisting of matrices with last row vector $(0,\cdots,0,1)$. Bernstein \cite{bernstein1984} proved that any $P_F$-invariant distribution on $Mat_{n,n}(F)$ must be $\GL_n(F)$-invariant when $F$ is non-archimedean. Baruch \cite{baruch2003annals} proved that any $P_F$-invariant eigendistribution (with respect to the center of the the universal enveloping algebra of $\mathfrak{gl}_n(F)$) is $\GL_{n}(F)$-invariant when $F$ is archimedean. It has been proved in \cite{AG2009selecta,sunzhu2012annals} that any $P_F$-invariant distribution is also  $\GL_{n}(F)$-invariant.
 It is expected that there is a more general phenomenon related to the mirabolic subgroup $P_F$. Let $H_{p,n-p}=\GL_p(F)\times \GL_{n-p}(F)$. Gurevich \cite{maxmirabolic} investigeted the role of the mirabolic subgroup $P_F$ of $\GL_n(F)$ on the symmetric variety $\GL_n(F)/H_{p,n-p}$ when $F$ is non-archimedean. Then Gurevich proved that any $H_{1,n-1}\cap P_F$-invariant linear functional on an $H_{1,n-1}$-distinguished irreducible smooth representation of $\GL_n(F)$ is also $H_{1,n-1}$-invariant (see \cite[Theorem 1.1]{maxmirabolic}). It is expected that it holds for all $H_{p,n-p}$. 
The case when $n-p=p+1$ has been verified in \cite{lu2020} if $F$ is non-archimedean (see \cite[Theorem 6.3]{lu2020}). 
Let $P_E$ denote the mirabolic subgroup of $\GL_n(E)$.
Then $P_E\cap\GL_n(F)=P_F$. 
Offen  and Kemarsky  proved that any $P_E\cap \GL_n(F)$-invariant linear functional on a $\GL_n(F)$-distinguished irreducible smooth representation of $\GL_n(E)$ is also $\GL_n(F)$-invariant. (See \cite[Theorem 3.1]{offen2011} for the p-adic case and \cite[Theorem 1.1]{kemarsky} for the archimedean case.)
This paper studies the role of the mirabolic subgroup $P$ of $\GL_{2n}(F)$ on the symmetric variety $\GL_{2n}(F)/\GL_n(E)$ and $P\cap \GL_{n}(E)=P_E$.

There is a group embedding $\GL_{2n}(F)\hookrightarrow\GL_{2n}(E)$ such that each element in the image of $\GL_{2n}(F)$ is of the form
\[\begin{pmatrix}
A& B\\ \bar{B}&\bar{A}
\end{pmatrix} \]
where $A,B\in Mat_{n,n}(E)$ and 
$ A\mapsto \bar{A}$ denotes the Galois action on $A$.
 Let $\theta$ be the involution of $\GL_{2n}(F)$ given by
\[ \theta:g\mapsto \begin{pmatrix}
\delta\\&-\delta
\end{pmatrix}g\begin{pmatrix}
\delta\\&-\delta
\end{pmatrix}^{-1} \]
for $g\in\GL_{2n}(F)$.
  Then the fixed points of $\theta$ in $\GL_{2n}(F)$ coincide with $\GL_n(E)$. Denote by $\mathfrak{gl}_{2n}(F)$ the Lie algebra of $\GL_{2n}(F)$. Under the above embedding $\GL_{2n}(F)\hookrightarrow \GL_{2n}(E)$, any $g$ in $\mathfrak{gl}_{2n}(F)$ is of the form
  \[\begin{pmatrix}
  a&b\\ \bar{b}&\bar{a}
  \end{pmatrix} \]
where $a,b\in Mat_{n,n}(E)$. Let $\mathfrak{p}$ denote the Lie algebra of $P$ which is given by
\[\Biggl\{ \begin{pmatrix}
a&b\\ \bar{b}&\bar{a}
\end{pmatrix}:\begin{matrix}
a=(a_{i,j}),b=(b_{i,j})\mbox{ for }i,j\in\{1,2,\cdots,n \}\\
a_{n,j}=\bar{b}_{n,j}\mbox{ for }j\in\{1,2,\cdots,n\}
\end{matrix} \Biggr\}. \]
Then $\mathfrak{p}\cap \mathfrak{gl}_n(E)=\mathfrak{p}_E$, where $\mathfrak{p}_E$
is the Lie algebra of $P_E=P\cap\GL_{n}(E)$.

 The main result in this paper is the following:
\begin{thm}
	\label{thmA} 
	Any $P_E$-invariant linear functional on a $\GL_n(E)$-distinguished irreducible smooth admissible representation $\pi$ of $\GL_{2n}(F)$ is also $\GL_n(E)$-invariant.
\end{thm} 
The geometry of closed $\GL_{n}(E)$-orbits on the symmetric space $\GL_{2n}(F)/\GL_{n}(E)$ is well known due to Guo \cite{guo1997unique} and Carmeli \cite{carmeli2015stability}. Then we will use the Harish-Chandra descent techniques developed in \cite{dima2009duke} to show the following identity of distributions
\begin{equation}\label{A}
\mathscr{D}(\GL_{2n}(F)/\GL_{n}(E))^{\GL_{n}(E)\cap P}=\mathscr{D}(\GL_{2n}(F)/\GL_{n}(E))^{\GL_{n}(E)} \end{equation}
(see \S4.1). Together with the injective map
\[A_\pi:\pi^\ast\otimes(\pi^\vee)^\ast\longrightarrow \mathscr{D}(\GL_{2n}(F)), \]
 \eqref{A} will lead to a proof of Theorem \ref{thmA}.  
 In fact, we will prove a slightly stronger result that any element in $\mathscr{D}(\GL_{2n}(F)/\GL_{n}(E))^{\GL_{n-1}(E)}$ is invariant under $\sigma$ where the action of $\sigma$ (order $2$) is given by $$\sigma:\begin{pmatrix}
 	A&B\\ \bar{B}&\bar{A}
 \end{pmatrix}\mapsto \begin{pmatrix}
 \bar{A}^t&B^t\\ \bar{B}^t&A^t
\end{pmatrix}$$
for any $\begin{pmatrix}
	A&B\\ \bar{B}&\bar{A}
\end{pmatrix}\in \GL_{2n}(F)$. Here we identify the symmetric variety $\GL_{2n}(F)/\GL_n(E)$ with the space of matrices
 \[X_n=\{g\in\GL_{2n}(F):g\theta(g)=1 \} \]
 and the transpose acts on $X_n$. Thus $\sigma$ acts on $\mathscr{D}(X_n)=\mathscr{D}(\GL_{2n}(F)/\GL_{n}(E))$ as well.
 \begin{thm}\label{thm:group}
 	One has $\mathscr{D}(X_n)^{\GL_n(E)\cap P }=\mathscr{D}(X_n)^{\GL_{n}(E)}$.
 \end{thm}
The key idea in the proof of Theorem \ref{thm:group} is to reduce a question on the distribution spaces of $X_n=\GL_{2n}(F)/\GL_{n}(E)$
to that of distributions on its tangenet space.

We may identify the linear version of $X_n=\GL_{2n}(F)/\GL_n(E)$ with the space of matrices
\[\begin{pmatrix}
&x\\ \bar{x}
\end{pmatrix} \]
for $x\in Mat_{n,n}(E)$ (see \cite{guo1997unique}), denoted by $L_n$. 
Let $\mathscr{C}(L_n)$ denote the {tempered} generalized functions on $L_n$.    Let $\GL_n(E)$ act on $L_n$ by the twisted conjugation, i.e., 
\[ \begin{pmatrix}
A\\&\bar{A}
\end{pmatrix}\cdot \begin{pmatrix}
&x\\ \bar{x}&
\end{pmatrix}=\begin{pmatrix}
& Ax\bar{A}^{-1}\\
\bar{A}\bar{x}A^{-1}&
\end{pmatrix} \]
for $A\in\GL_n(E)$ and let $\mathfrak{gl}_n(E)$
act on $L_n$ by its differential. More precisely,
\[\begin{pmatrix}
a\\&\bar{a}
\end{pmatrix}\cdot \begin{pmatrix}
&x\\ \bar{x}
\end{pmatrix}=\begin{pmatrix}
&ax-x\bar{a}\\
\bar{a}\bar{x}-\bar{x}a
\end{pmatrix} \]
for $a,x\in Mat_{n,n}(E)$.
 Then the following theorem is the Lie algebra version of Theorem \ref{thm:group}.
\begin{thm}\label{thmB}
	One has \[\mathscr{C}(L_n)^{\GL_n(E)\cap P }=\mathscr{C}(L_n)^{\GL_n(E)}. \]
\end{thm}
\begin{rem}
	Here we study the tempered generalized functions space $\mathscr{C}(L_n)$ instead of the generalized functions or distributions $\mathscr{D}(L_n)$ on $L_n$ because we will use the Fourier transform on $\mathscr{C}(L_n)$. Moreover, there is not so much difference between $\mathscr{C}(L_n)$ and $\mathscr{D}(L_n)$ due to \cite[Theorem 4.0.2]{dima2009duke}.
\end{rem}

We give a brief introduction to the proof of Theorem \ref{thmB}. We will use the result of Aizenbud-Gourevitch (see Theorem \ref{thm2.1})
to reduce the problems on the tempered generalized functions on $L_n$ to those on the tempered generalized functions supported on the nilpotent cone. 
If $F$ is non-archimedean,
then we will pick up an $\mathfrak{sl}_2(F)$-triple $\{\mathbf{h},\mathbf{e},\mathbf{f} \}$ (see \eqref{sl}) and use Chen-Sun's method \cite{sun2020} to study some special nilpotent orbits $\mathcal{O}\owns\mathbf{e}$. If $F=\mathbb{R}$, then we will use the machine of $D$-modules to show the vanishing theorem.
Note that $\GL_{n-1}(E)$ is a proper subgroup of $P_E=\GL_{n}(E)\cap P$.
It turns out that each $\GL_{n-1}(E)$-invariant tempered generalized function on $L_n$ supported on $\mathcal{O}$ is invariant under transposition (see Theorem \ref{vanishing:I}), which  implies Theorem \ref{thmB}.

The paper is organized as follows. In \S2, we introduce some notations from algebraic geometry. Then we will use Chen-Sun's method (resp. the machine of D-modules) to prove Theorem \ref{thmB} when $F$ is non-archimedean (resp. $F=\mathbb{R}$) in \S3. In \S4.1, we will  give a proof of Theorem \ref{thm:group}. The proof of Theorem \ref{thmA} will be given in \S\ref{subsect:A}.

\section{Preliminaries and notation}

Let $X$ be an $\ell$-space (i.e. locally compact totally disconnected topological spaces) if $F$ is non-archimedean or a Nash manifold (see \cite{dima2009duke}) if $F=\mathbb{R}$. 
Let $\mathscr{C}(X)$ denote the tempered generalized functions on $X$. Let a reductive group $G(F)$
act on an affine variety $X$. Let $x\in X$ such that its orbit $G(F)x$ is closed in $X$. We denote the normal bundle by $N_{G(F)x,x}^X$. Let $$G_x:=\{g\in G(F)|gx=x \}$$ be the stalizer subgroup of $x$.
\begin{thm}\cite[Theorem 3.1.1]{dima2009duke}
	Let $G(F)$ act on a smooth affine variety $X$. Let $\chi$ be a character of $G(F)$. Suppose that for any closed orbit $Gx$ in $X$, we have
	\[\mathscr{C}(N^X_{G(F)x,x})^{G_x,\chi}=0. \]
	Then \[\mathscr{C}(X)^{G(F),\chi}=0. \]\label{thm:duke}
\end{thm}
If $X$ is a finite
dimensional (complex) representation of $G(F)$, then we denote the nilpotent cone in $X$ by
\[\Gamma(X):=\{ x\in X|\overline{G(F)x}\owns 0 \}. \] 
Let $Q_G(X):=X/X^G$ and $R_G(X):=Q_G(X)- \Gamma(X)$.
\begin{thm}\cite[Corollary 3.2.2]{dima2009duke}
	\label{thm2.1}
	Let $X$ be a finite dimensional representation of a reductive group $G(F)$. Let $K\subset G(F)$ be an open subgroup and let $\chi$ be a character of $K$. Suppose that for any closed orbit $G(F)x$ such that
	\[\mathscr{C}(R_{G_x}(N_{G(F)x,x}^X))^{K_x,\chi}=0 \]
	we have \[\mathscr{C} (Q_{G_x}(N^X_{G(F)x,x}))^{K_x,\chi}=0. \]
	Then $\mathscr{C}(X)^{K,\chi}=0$.
\end{thm}

\subsection{D-modules and singular support} In this subsection, assume that $F=\mathbb{R}$.
Let $X$ be a Nash manifold. Denote by $S(X)$ the space of Schwartz functions on $X$. Denote by $\mathscr{C}(X)$ the linear dual space to $S(X)$, i.e. the tempered generalized functions on $X$. All the materials in this subsection come from \cite{dima2009duke,AG2009selecta,aizenbud2013partial}.

\subsubsection{Coisotropic variety}
Let $M$ be a smooth algebraic variety and $\omega$
be a symplectic form on it. Let $Z\subset M$ be an algebraic subvariety. We call it $M$-\textbf{coisotropic} if $T_zZ\supset (T_zZ)^\perp$ for a generic smooth point $z\in Z$, where $(T_zZ)^\perp$ denotes the orthogonal complement to $T_zZ$ in $T_zM$ with respect to $\omega$. Note that every non-empty $M$-coisotropic variety is of dimension at least $\frac{1}{2}\dim M$. For a smooth algebraic variety $X$, we always consider the standard symplectic form on the cotangent bundle $T^\ast X$. Also, we denote by $p_X:T^\ast X\rightarrow X$ the standard projection. 

\begin{lem}\cite[Lemma 3.0.7]{aizenbud2013partial}
	Let $X$ be a smooth algebraic variety. Let a group $G$ act on $X$ which induces an action on $T^\ast X$. Let $S\subset T^\ast X$ be a $G$-invariant subvariety. Then the maximal $T^\ast X$-coisotropic subvariety of $S$ is also $G$-invariety.
\end{lem}
Let $Y$ be a smooth algebraic variety. Let $Z\subset Y$ be a smooth subvariety.
Let $R\subset T^\ast Y$ be any subvariety. We define the restriction
\[R|_Z:=i^\ast(R) \]
of  $R$ to $Z$ in $T^\ast Z$, where $i:Z\rightarrow Y$ is the embedding. 
\begin{lem}\cite[Lemma 3.0.9]{aizenbud2013partial}
	Let $R\subset T^\ast Y$ be a coisotropic subvariety. Assume that any smooth point $z\in Z\cap p_Y(R)$ is also a smooth point of $p_Y(R)$ and we have $T_z(Z\cap p_Y(R))=T_z(Z)\cap T_z(p_Y(R))$. Then $R|_Z$ is $T^\ast Z$-coisotropic.
\end{lem}
\begin{coro}\cite[Corollary 3.0.11]{aizenbud2013partial}
	Let $Y$ be a smooth algebraic variety. Let an algebraic group $H$ act on $Y$. Let $q:Y\rightarrow B$ be an $H$-equivariant morphism. Let $\mathcal{O}\subset B$ be an orbit. Consider the natural action of $H$ on $T^\ast Y$ and let $R\subset T^\ast Y$ be an $H$-invariant subvariety. Suppose that $p_Y(R)\subset q^{-1}(\mathcal{O})$. Let $x\in \mathcal{O}$. Denote $Y_x:=q^{-1}(x)$. Then if $R$ is $T^\ast Y$-coisotropic then $R|_{Y_x}$ is $T^\ast(Y_x)$-coisotropic. Thus if $R|_{Y_x}$
	has no (non-empty) $T^\ast(Y_x)$-coisotropic subvarieties then $R$ has no (non-empty) $T^\ast Y$-coisotropic subvariety.
\end{coro}

\subsubsection{Singular support}
Let $X$ be a smooth algebraic variety. Let $D_X$ denote the algebra of polynomial differential operators on $X$.  Let $Gr D_X$ be the associated graded algebra of $D_X$. Then $Gr D_X\cong \mathcal{O}(T^\ast X)$.

Let $\xi\in\mathscr{C}(X)$. Denoted by $SS(\xi)$ the singular support of the right $D_X$-module generated by $\xi$. Then $SS(\xi)\subset T^\ast X$ is nothing but the zero set of $Gr(Ann_{D_X}\xi)$ where
$Ann_{D_X}\xi$ is the annihilator of $\xi$. (See \cite[Appendix B]{AG2009selecta} for more details.)

Let $V$ be a vector space over $F$. 
Let $B$ be a non-degenerate bilinear form on $V$. Then $B$ defines  Fourier transform with respect to the self-dual Haar measure on $V$, denoted by $\mathfrak{F}_V$. For any Nash manifold $M$, we also denote by
\[\mathfrak{F}_V:\mathscr{C}(M\times V)\rightarrow \mathscr{C}(M\times V) \]
the fiberwise Fourier transform or partial Fourier transform. Consider $B$ as a map $B:V\rightarrow V^\ast$. Identify $T^\ast(X\times V)$ with $T^\ast X\times V\times V^\ast$. We define
\[F_V:T^\ast(X\times V)\rightarrow T^\ast(X\times V) \]
by $F_V(x,v,\phi)=(x,-B^{-1}\phi,Bv)$.
\begin{prop}\cite[\S2.3]{AG2009selecta}
	\begin{enumerate}[(i)]
		\item Let $\xi\in \mathscr{C}(X)$. Then the Zariski closure of $Supp(\xi)$ is $p_X(SS(\xi))$.
		\item Let an algebraic group $G$ act on $X$. Let $\mathfrak{g}$ denote the Lie algebra of $G$.
		Let $\xi\in \mathscr{C}(X)^{G}$. Then $SS(\xi)\subset \{(x,\phi)\in T^\ast X| \phi(\alpha(x))=0 \mbox{ for all }\alpha\in\mathfrak{g} \}$.
		\item Let $(V,B)$ be a quadratic space. Let $X$ be a smooth algebraic variety. 
		  Let $Z\subset X\times V$ be a closed subvariety, invariant with respect to homotheties in $V$. Suppose that $Supp(\xi)\subset Z$. Then $SS(\mathfrak{F}_V(\xi))\subset F_V(p^{-1}_{X\times V}(Z))$.
		\item Let $X$ be a smooth algebraic variety. Let $\xi\in\mathscr{C}(X)$. Then $SS(\xi)$ is coisotropic.
	\end{enumerate}
\end{prop}
\subsubsection{Distributions on non-distinguished nilpotent orbits}
	Let $V$ be an algebraic finite dimensional representation of a reductive group $G$. Let $\Gamma(V)$
	be the nilpotent cone of $V$.
\begin{defn}
 Suppose that there is a finite number of $G$-orbits in $\Gamma(V)$. Let $x\in \Gamma(V)$. We call it $G$-distinguished if its conormal bundle $CN_{Gx,x}^{Q(V)}\subset\Gamma(V^\ast)$. We will call a $G$-orbit $G$-distinguished if all its elements are $G$-distinguished.
\end{defn}
In the case when $G =\GL_n(\mathbb{R}) $ and $V =Mat_{n,n}(\mathbb{R}) $ the set of $G$-distinguished elements is exactly the set of regular nilpotent elements.
\begin{prop}\cite[Proposition 4.3.4]{aizenbud2013partial}
	Let $W:=Q(V)$ and let $A$ be the set of non-distinguished elements in $\Gamma(V)$. Then there are no non-empty $W\times W^\ast$-coisotropic subvarieties of $A\times \Gamma(V^\ast)$.
\end{prop}
\begin{coro}\cite[Corollary 4.3.5]{aizenbud2013partial}
	Let $\xi\in \mathscr{C}(W)$ and suppose that $Supp(\xi)\subset \Gamma(V)$ and $Supp(\mathfrak{F}(\xi))\subset \Gamma(V^\ast)$. Then the set of distinguished elements in $Supp(\xi)$ is dense in $Supp(\xi)$.
\end{coro}

\section{A vanishing result of generalized functions}
In this section, we shall prove Theorem \ref{thmB}.
Recall that 
$$L_{n}=Mat_{n,n}(E)\cong\Big\{\begin{pmatrix}
0&x\\ \bar{x}&0
\end{pmatrix}:x\in Mat_{n,n}(E) \Big\}\subset \mathfrak{gl}_{2n}(F).$$ Let $H_{n}:=\GL_{n}(E)$.
Denote $\tilde{H}_{n}:=H_{n}\rtimes\langle\sigma\rangle$ where $\sigma$ acts on $H_{n}$ by the involution 
$$\begin{pmatrix}
a\\&\bar{a}
\end{pmatrix}\mapsto\begin{pmatrix}
(\bar{a}^{-1})^t\\&(a^{-1})^t
\end{pmatrix}.$$
Note that $\sigma$ is of order $2$.
The group $\tilde{H}_{n}$ acts on $L_{n}$ by
\[\begin{pmatrix}
a\\&\bar{a}
\end{pmatrix}\cdot \begin{pmatrix}
&x\\ \bar{x}
\end{pmatrix}=\begin{pmatrix} &ax\bar{a}^{-1}\\
\bar{a}\bar{x}a^{-1}
\end{pmatrix} \]
and $$\sigma\cdot \begin{pmatrix}
	&x\\ \bar{x}
\end{pmatrix}=\begin{pmatrix}
&x^t\\ \bar{x}^t
\end{pmatrix}$$
for $x\in Mat_{n,n}(E)$. Let $\chi$ be the sign character of $\tilde{H}_{n}$, i.e.
$\chi|_{H_{n}}$ is trivial and
\[\chi(\sigma)=-1. \]
The group $H_{n-1}$ is viewed as a subgroup  of $H_n$ via the embedding
\[ g\mapsto \begin{pmatrix}
g&0\\0&1
\end{pmatrix} \]
for $g\in\GL_{n-1}(E)$.
 Then $H_{n-1}$ is a proper subgroup of $P_E=P\cap \GL_{n}(E)$.
\begin{thm}\label{vanishing:I}
	We have $\mathscr{C}(L_{n})^{\tilde{H}_{n-1},\chi}=0$.
\end{thm}
 Then Theorem \ref{thmB} follows from Theorem \ref{vanishing:I} due to the fact that the subgroups $P_E$ and its transpose $P_E^t$ generate the whole group $\GL_n(E)$.

Consider the decomposition $$L_n=L_{n-1}\oplus V\oplus V^\ast\oplus E$$
 of $\tilde{H}_{n-1}$-spaces, where $\tilde{H}_{n-1}$ acts on $E$ trivially and acts on $L_{n-1}\oplus V\oplus V^\ast$ via
 \[g\cdot (x,v,v^\ast)=(gx\bar{g}^{-1},gv,v^\ast\bar{g}^{-1} ) \]
 for $x\in L_{n-1}=Mat_{n-1,n-1}(E),v\in V$ and $v^\ast\in V^\ast$,
  $V^\ast$ is the linear dual space of $V$ and $\dim_E V=n-1$.
  Denote by
  \[\mathcal{N}_n:=\Bigg\{(x,v,v^\ast)\in L_{n-1}\oplus V\oplus V^\ast\Big|\begin{matrix} (x\bar{x})^{n-1}=0 \mbox{ and }  v^\ast(\bar{x}x)^k\bar{v}=0=v^\ast\bar{x}(x\bar{x})^kv \mbox{ for all non-negative integer }k
  \end{matrix} \Bigg\} \]
  the nilpotent cone in $L_{n-1}\oplus V\oplus V^\ast$. (See \cite[\S6.1]{aizenbud2013partial}.)
  Let $\mathcal{O}$ be an $H_{n-1}$-orbit in $\mathcal{N}_n$.
  Denote by $\mathscr{C}_{\mathcal{O}}(L_{n-1}\oplus V\oplus V^\ast )$ the space of the {tempered} generalized functions on $(L_{n-1}\oplus V\oplus V^\ast)\setminus \partial\mathcal{O}$ with support
  in $\mathcal{O}$, where $\partial\mathcal{O}$ is the complement of $\mathcal{O}$ in its closure in $L_{n-1}\oplus V\oplus V^\ast$. (See \cite[Notation 2.5.3]{dima2009duke}.) We will use similar notation without further explaination.
For the proof of Theorem \ref{vanishing:I}, we will prove the following.
\begin{thm}\label{key:vanishing}
	One has
	\[\mathscr{C}_{\mathcal{N}_n}(L_{n-1}\oplus V\oplus V^\ast)^{\tilde{H}_{n-1},\chi}=0. \]
\end{thm}
We will show that Theorem \ref{vanishing:I} follows from Theorem \ref{key:vanishing} later.

Define a non-degenerate symmetric $F$-bilinear form on $\mathfrak{gl}_{2n}(F)$ by
\[\langle z, w\rangle_{\mathfrak{gl}_{2n}(F)}:=\mbox{the trace of }zw\mbox{ as an }F\mbox{-linear operator}. \]
Note that the restriction of this bilinear form on $L_{n}$ is still non-degenerate. Fix a non-trivial unitary character $\psi$ of $F$. Denote by
\[\mathfrak{F}:\mathscr{C}(L_{n})\longrightarrow \mathscr{C}(L_{n}) \]
the Fourier transform which is normalized such that for every Schwartz function $\varphi$ on $L_{n}$,
\[\mathfrak{F}(\varphi)(z)=\int_{L_{n}}\varphi(w)\psi(\langle z,w \rangle_{\mathfrak{gl}_{2n}(F)})d w  \]
for $z\in L_{n}$, where $d w$ is the self-dual Haar measure on $L_{n}$. If $L_{n}$ can be decomposed into a direct sum of two quadratic subspaces $U_1\oplus U_2$ such that each $U_i$ is non-degenerate with respect to $\langle-,-\rangle|_{U_i}$, then we may define the partial Fourier transform
\[\mathfrak{F}_{U_1}(\varphi)(x,y)=\int_{U_1}\varphi(z,y) \psi(\langle x,z\rangle|_{U_1}) dz \]
for $x\in U_1,y\in U_2$ and $\varphi\in \mathscr{C}(U_1\oplus U_2)$. Similarly for $\mathfrak{F}_{U_2}(\varphi)$.
It is clear that the Fourier transform $\mathfrak{F}$ intertwines the action of $\tilde{H}_{n-1}$. Thus we have the following lemma.
\begin{lem}\label{lem:intertwin}
	The Fourier transform $\mathfrak{F}$ preserves the space $\mathscr{C}(L_{n-1}\oplus V\oplus V^\ast)^{\tilde{H}_{n-1},\chi}$.
\end{lem}

\subsection{Reduction within the null cone}
Recall that
\[\mathcal{N}_n:=\Bigg\{(x,v,v^\ast)\in L_{n-1}\oplus V\oplus V^\ast\Big|\begin{matrix} (x\bar{x})^{n-1}=0 \mbox{ and }  v^\ast (\bar{x}x)^k\bar{v}=0=v^\ast\bar{x}(x\bar{x})^kv \mbox{ for all non-negative integer }k
\end{matrix} \Bigg\} \]
 is the nilpotent cone in $L_{n-1}\oplus V\oplus V^\ast$.  Let 
$$\mathcal{N}:=\{x\in Mat_{n-1,n-1}(E):x\bar{x}\mbox{ is nilpotent} \}.$$

\begin{lem}\label{lem:guo}
	\cite[Lemma 2.3]{guo1997unique} For any $x$ in $\mathcal{N}$, there exists a $g$ in $H_{n-1}$
	such that the twisted conjugate $gx\bar{g}^{-1}$ of $x$ is in its Jordan normal form.
\end{lem}
Following \cite[Proposition 3.9]{sun2020},  we shall prove the following proposition when $F$ is non-archimedean in this subsection.
\begin{prop}\label{fourier}
	Let $f$ be an $H_{n-1}$-invariant generalized function on $L_{n-1}\oplus V\oplus V^\ast$ such that $f$, its Fourier transform $\mathfrak{F}(f)$ and its partial Fourier transforms $\mathfrak{F}_{V\oplus V^\ast}(f),\mathfrak{F}_{L_{n-1}}(f)$ are all supported on $\mathcal{N}_n$. Then $f=0$.
\end{prop}
\begin{rem} 
	We will postpone the proof of Proposition \ref{fourier} when $F=\mathbb{R}$ until the next subsection, which involves the machine of D-modules.
\end{rem}

Let $\mathcal{O}$ be an $H_{n-1}$-orbit in $\mathcal{N}_n$. Pick $(\mathbf{e},v_0,v_0^\ast)\in\mathcal{O}$. Then $\mathbf{e}\in \mathcal{N}$.  Moreover, we may assume that $\mathbf{e}=\begin{pmatrix}
	&x\\ x
\end{pmatrix}$ with $x\in Mat_{n-1,n-1}(F)$ due to Lemma \ref{lem:guo}. Recall that every $\mathbf{e}\in L_{n-1}$
can be extended to an $\mathfrak{sl}_2$-triple $\{\mathbf{h},\mathbf{e},\mathbf{f} \}$
(see \cite[Proposition 4]{kostant}) in the sense that
\begin{equation}\label{sl}
[\mathbf{h},\mathbf{e}]=2\mathbf{e},~[\mathbf{h},\mathbf{f}]=-2\mathbf{f}\mbox{  and  }[\mathbf{e},\mathbf{f}]=\mathbf{h}
\end{equation}
where $\mathbf{f}\in \mathcal{N}$ and $\mathbf{h}\in \mathfrak{h}_{n-1}$, where $\mathfrak{h}_{n-1}=\mathfrak{gl}_{n-1}(E)$
is the Lie algebra of $H_{n-1}$. Furthermore, we may assume that $\mathbf{f},\mathbf{h}\in Mat_{n-1,n-1}(F)$ due to Lemma \ref{lem:guo}. Let $L_{n-1}^\mathbf{f}$ denote the set which consists of elements in $L_{n-1}$
annihilated by $\mathbf{f}$ under the adjoint action of the triple $\{\mathbf{h},\mathbf{e},\mathbf{f} \}$
on $L_{n-1}$. Then
\[L_{n-1}=[\mathfrak{h}_{n-1},\mathbf{e}]\oplus L_{n-1}^\mathbf{f}. \]


Let $F^\ast$ act on $\mathscr{C}(L_{n-1}\oplus V\oplus V^\ast)$ by
\[(t\cdot f)(x,v,v^\ast)=f(t^{-1}x,t^{-1}v,t^{-1}v^\ast) \]
for $t\in F^\times,x\in L_{n-1},v\in V,v^\ast\in V^\ast$ and $f\in\mathscr{C}(L_{n-1}\oplus V\oplus V^\ast)$. The orbit $\mathcal{O}$ is invariant  under dilation and so $F^\times$ acts on $\mathscr{C}_{\mathcal{N}_{n}}(L_{n-1}\oplus V\oplus V^\ast)^{H_{n-1}}$ as well.

Suppose that the $\mathfrak{sl}_2(F)$-triple \eqref{sl} integrates to an algebraic homomorphism
\[\SL_2(F)\longrightarrow \GL_{2n-2}(F) .\]
Denote by $D_t$ the image of $\begin{pmatrix}
t\\&t^{-1}
\end{pmatrix}$ in $H_{n-1}\cap \GL_{2n-2}(F)$. Let $T$ be a closed subgroup in $H_{n-1}\times F^\times$ which fixes the element $\mathbf{e}$.
Note that, T can be described explicitly, i.e.
\[T=\{(D_t,t^{-2})\in H_{n-1}\times F^\times|t\in F^\times \}. \]
Define a quadratic form on $V\oplus V^\ast$ as follows:
\[(v,v^\ast)\mapsto tr_{E/F}(v^\ast(\bar{v})) \]
for $v\in V$ and $v^\ast\in V^\ast$
which induces an $F$-bilinear form $\langle-,-\rangle$ on $(V\oplus V^\ast)\times (V\oplus V^\ast)$.
Define
\[V(\mathbf{e}):=\{(v,v^\ast)\in V\oplus V^\ast|(\mathbf{e},v,v^\ast)\in\mathcal{O}\mbox{ and }\langle h\cdot(v,v^\ast),(v,v^\ast)\rangle=0\mbox{ for all }h\in\langle D_t\rangle \}. \]
The following lemma is similar to \cite[Lemma 3.13]{sun2020}.
\begin{lem}\label{lem:key}
	Let $\eta$ be an eigenvalue for the action of $F^\times$ on $\mathscr{C}_\mathcal{O}(L_{n-1}\oplus V\oplus V^\ast)^{H_{n-1}}$. Let $|-|$ denote the absolute value of $F^\times$. Then
	$\eta^2=|-|^{tr(2-\mathbf{h})|_{L_{n-1}^\mathbf{f}}+4(n-1)}$.
\end{lem}
\begin{proof}
	Consider the map
		\begin{equation}
	\label{submersion}
	H_{n-1}\times F^\times\times (L_{n-1}^\mathbf{f}\oplus V\oplus V^\ast)\longrightarrow L_{n-1}\oplus V\oplus V^\ast 
	\end{equation}
	via $(h,\xi,v,v^\ast)\mapsto h.(\mathbf{e}+\xi+v+v^\ast)$ for $\xi\in L_{n-1}^\mathbf{f},h\in H_{n}\times F^\times,v\in V$ and $v^\ast\in V^\ast$, which is
	submersive at every point of $H_{n-1}\times F^\times\times\{(0,v_0,v_0^\ast)\}$. Moreover,  $H_{p,p}\times F^\times\times\{(0,v_0,v_0^\ast)\}$ is open in the inverse image of $\mathcal{O}=(H_{n-1}\times F^\times)\cdot (\mathbf{e},v_0,v_0^\ast)$ under the map \eqref{submersion}. (See \cite[Page 18]{sun2020}.)
	Thanks to \cite[Lemma 2.7]{jiang2011trans},
	the restriction map yields an injective linear map
	\[ \mathscr{C}_{\mathcal{O}}(L_{n-1}\oplus V\oplus V^\ast)^{H_{p,p}\times F^\times,\mathbf{1}\times\eta}\longrightarrow
	\mathscr{C}_{\{0\}\times E(\mathbf{e})}(L_{n-1}^\mathbf{f}\oplus V\oplus V^\ast)^{T,\mathbf{1}\times\eta|_T} \]
	where $\mathbf{1}\times\eta|_{T}((D_t,t^{-2}))=\eta(t)^{-2}$ and
	$$E(\mathbf{e}):=\{(v,v^\ast)\in V\times V^\ast|v^\ast(x_0)^{2k}\bar{v}=0=v^\ast x_0^{2k+1}v\mbox{ for all non-negative integers }k \}$$
	for $\mathbf{e}=\begin{pmatrix}
	&x_0\\ x_0
	\end{pmatrix}\in L_{n-1}$.
	It is easy to see that the representation $\mathscr{C}_{\{0\}}(L_{n-1}^\mathbf{f})$
	of $T$ is complete reducible and every eigenvalue has the form
	\[ (D_t,t^{-2})\mapsto |t|^{tr(\mathbf{h}-2)|_{L_{n-1}^\mathbf{f}}}.\]
	Thus $$\eta(t)^2=|t|^{tr(2-\mathbf{h})|_{L_{n-1}^\mathbf{f}}}\gamma^{-1}(t)$$ for any $t\in F^\times$, where $\gamma$ is an eigenvalue for the action of $T$ on $\mathscr{C}_{E(\mathbf{e})}(V\oplus V^\ast)$.
	In order to compute $\gamma$, we will restrict $\gamma$ to a smaller subspace $\mathscr{C}_{V(\mathbf{e})}(V\oplus V^\ast)$ of $\mathscr{C}_{E(\mathbf{e})}(V\oplus V^\ast)$. 
	\par
	Define a symplectic form on $(V\oplus V^\ast)\times (V\oplus V^\ast)$ as follow
	\[ < (x_1,y_1),(x_2,y_2)>:=\langle x_1,y_2\rangle-\langle y_1,x_2\rangle   \]
	where $x_i,y_i\in V\oplus V^\ast$. Then $V\oplus V^\ast$ is a maximal isotropic subspace.
	Consider the Weil representation of $\mathrm{Mp}_{4n-4}(F)=\mathrm{Mp}((V\oplus V^\ast)\times(V\oplus V^\ast),<-,->)$ on $S(V\oplus V^\ast)$.
	Under the Weil representation $\omega_\psi$,
	\[\begin{cases}
	\omega_\psi\begin{pmatrix}
	A\\& (A^t)^{-1}
	\end{pmatrix}\varphi(x)=|\det A|^{1/2}\varphi(A^{-1}x),&\mbox{ for }A\in\GL_{2n-2}(F),\\
	\omega_\psi\begin{pmatrix}
	\mathbf{1}_{2n-2}&N\\&\mathbf{1}_{2n-2}
	\end{pmatrix}\varphi(x)=\psi(\langle Nx,x\rangle)\varphi(x),&\mbox{ for }N=N^t,
	\end{cases}
	\]
	for $\varphi\in S(V\oplus V^\ast)$ and $x\in V\oplus V^\ast$. We may extend $\omega_\psi$ from the Schwartz  space $S(V\oplus V^\ast)$ to the tempered generalized function space $\mathscr{C}(V\oplus V^\ast)$.
	Note that
	\[\begin{pmatrix}
	X\\& X^{-1}
	\end{pmatrix}=\begin{pmatrix}
	\mathbf{1}_n&-X\\&\mathbf{1}_n
	\end{pmatrix}\begin{pmatrix}
	\mathbf{1}_n\\ X^{-1}&\mathbf{1}_n
	\end{pmatrix}\begin{pmatrix}
	\mathbf{1}_n&1-X\\&\mathbf{1}_n
	\end{pmatrix}\begin{pmatrix}
	\mathbf{1}_n\\-\mathbf{1}_n&\mathbf{1}_n
	\end{pmatrix}\begin{pmatrix}
	\mathbf{1}_n&\mathbf{1}_n
	\\&\mathbf{1}_n	\end{pmatrix}  \]
	holds for any $X\in\GL_{n}(F)$. Here we only need  the case that $X$ is a diagonal matrix. 
	Denote
	$D_t=\begin{pmatrix}
	A_t\\&B_t
	\end{pmatrix}$ and $X_t=\begin{pmatrix}
	A_t\\& A_t^{-1}
	\end{pmatrix}$. 
	Then the action of $D_t$ on $V\oplus V^\ast$ is given by
	$$(v,v^\ast)\mapsto (A_tv,v^\ast A_t^{-1}).$$ 
	It is obvious that
	\begin{equation*}
	\begin{split}\omega_\psi
	\begin{pmatrix}
	\mathbf{1}_{2n-2}& X_t\\&\mathbf{1}_{2n-2}
	\end{pmatrix}f(v,v^\ast)&=\psi(\langle (A_tv,v^\ast A_t^{-1}),(v,v^\ast)\rangle)f(v,v^\ast)
	\\ 
	&=f(v,v^\ast)
	\end{split}
	\end{equation*}
	for any $f\in\mathscr{C}_{V(\mathbf{e})}(V\oplus V^\ast)$. Then
	$\begin{pmatrix}
	\mathbf{1}_{2n-2}&X_t\\ &\mathbf{1}_{2n-2}
	\end{pmatrix}$  acts on $\mathscr{C}_{V(\mathbf{e})}(V\oplus V^\ast)$
	trivially and so is $\begin{pmatrix}
	\mathbf{1}_{2n-2}\\ X_t^{-1}&\mathbf{1}_{2n-2}
	\end{pmatrix}$. 
	Thus $D_t$ does not contribute to $\gamma$. Therefore $\gamma$ has the form
	\[(D_t,t^{-2})\mapsto |t^{-2}|^{\cdot \frac{1}{2}\dim_F(V\oplus V^\ast)}=|t|^{4-4n} \]
	and so $\eta(t)^2=|t|^{tr(2-\mathbf{h})|_{L_{n-1}^\mathbf{f}}+4n-4}$ for any $t\in F^\times$.
\end{proof}
	 Consider $Mat_{n-1,n-1}(F)$ as a representation of the $\mathfrak{sl}_2(F)$-triple \eqref{sl}. Decompose it into irreducible representations
\[Mat_{n-1,n-1}(F)=\oplus_{i=1}^d V_i .\]
Let $\lambda_i$ be the highest weight of $V_i$. Note that $tr(2-\mathbf{h})|_{L_{n-1}^\mathbf{f}}$ is an integer.
\begin{lem}Assume $n\geq3$.
	One has 
	\begin{equation}\label{inequa}
		2(n-1)^2+3 <tr(2-\mathbf{h})|_{L_{n-1}^\mathbf{f}}< 4(n-1)^2
	\end{equation}
\end{lem}
\begin{proof}
	It is easy to see that $(n-1)^2=\sum_{i=1}^d(\lambda_i+1)=\sum_i\lambda_i+d$. Note that
	\[tr(2-\mathbf{h})|_{L_{n-1}^\mathbf{f}}=2(2d+\sum_i\lambda_i). \]
	Therefore $tr(2-\mathbf{h})|_{L_{n-1}^\mathbf{f}}-2(n-1)^2=2d>3$ due to the fact that $d\geq2$.
\end{proof}
Let $Q$ be a quadratic form on $L_{n-1}\oplus V\oplus V^\ast$ defined by
\[Q(x,v,v^\ast)=tr(x\bar{x})+tr_{E/F}v^\ast(\bar{v}) \]
for $x\in Mat_{n-1,n-1}(E)= L_{n-1},v\in V$ and $v^\ast\in V^\ast$. Denote by $Z(Q)$ the zero locus of $Q$  in $L_{n-1}\oplus V\oplus V^\ast$. Then $\mathcal{N}_{n}\subset Z(Q)\subset L_{n-1}\oplus V\oplus V^\ast$. Recall the following homogeneity result on tempered generalized functions due to Aizenbud-Gourevitch. (See
\cite[Theorem 5.1.5]{dima2009duke}. )
\begin{thm}
	\label{duke}
		Let $I$ be a non-zero subspace of $\mathscr{C}_{Z(Q)}(L_{n-1}\oplus V\oplus V^\ast)$ such that for every $f\in I$, one has that $\mathfrak{F}(f)\in I$ and $(\psi\circ Q)\cdot f\in I$ for all unitary character $\psi$ of $F$. Then $I$ is a completely reducible $F^\times$-subrepresentation of ${\mathscr{C}}(L_{n-1}\oplus V\oplus V^\ast)$, and it has an eigenvalue of the form $	|-|^{\frac{1}{2}\dim_F (L_{n-1}\oplus V\oplus V^\ast)}$.
\end{thm}

Now we are prepared to prove Proposition \ref{fourier} when $F$ is non-archimedean.
\begin{proof}[Proof of Proposition \ref{fourier} when $F$ is non-archimedean]
	Denote by $I$ the space of all tempered generalized functions $f$ on $L_{n-1}\oplus V\oplus V^\ast$ with the properties in Proposition \ref{fourier}. Assume by contradiction that $I$ is nonzero. If $n-1=1$, i.e. $n=2$, then $$\mathcal{N}_n\cong \{0\}\oplus (E\oplus\{0\}\cup \{0\}\oplus E)$$ and so
	Proposition \ref{fourier} follows from \cite[Lemma 6.3.4]{aizenbud2013partial} that if there exists an
	\[f\in\mathscr{C}_{\mathcal{N}_n}( E\oplus E\oplus E )^{E^\times} \]
such that both $\mathfrak{F}_{L_{n-1}}(f)$ and $\mathfrak{F}_{V\oplus V^\ast}(f)$ are supported on $\mathcal{N}_{n}$, then $f=0$.	Here the action of $E^\times$ is given by $$g\cdot(x,v,v^\ast)=(gx\bar{g}^{-1},gv,v^\ast \bar{g}^{-1})$$ where $g\in E^\times,x,v,v^\ast\in E$.
	 Assume that $n\geq3$. 
	Then by Lemma \ref{lem:key} and Theorem \ref{duke}, one has
	\[\dim_F(L_{n-1}\oplus V\oplus V^\ast)=tr(2-\mathbf{h})|_{L_{n-1}^\mathbf{f}}+4n-4 \]
	and so 
	\[tr(2-\mathbf{h})|_{L_{n-1}^{\mathbf{f}}}=2(n-1)^2 \]
	which contradicts the inequality \eqref{inequa}.
	This finishes the proof.
\end{proof}

\subsection{Proof of Proposition \ref{fourier} when $F=\mathbb{R}$} This subsection focuses on the proof of Proposition \ref{fourier} when $F=\mathbb{R}$. We will follow \cite[\S6]{aizenbud2013partial} to prove that $SS(f)$ is not coisotropic for any non-zero tempered generalized function $f$ satisfying the conditions in Proposition \ref{fourier}, which implies that $f$ must be zero.

Recall that $\mathcal{N}=\{x\in Mat_{n-1,n-1}(\mathbb{C})|x\bar{x} \mbox{ is nilpotent} \}$ and
\[\mathcal{N}_n=\{(x,v,v^\ast)\in L_{n-1}\oplus V\oplus V^\ast|x\in\mathcal{N},v^\ast(\bar{x}x)^k\bar{v}=0=v^\ast\bar{x}(x\bar{x})^kv\mbox{ for all non-negative integer }k \}. \]
Define 
$$S=\Big\{((A_1,v_1,v_1^\ast),(A_2,v_2,v_2^\ast))\Big|
\begin{matrix} (A_i,v_j,v_j^\ast)\in \mathcal{N}_n \mbox{ for any }i,j\in\{1,2\}\mbox{ and }\\
\alpha(A_1,v_1,v_1^\ast)\perp (A_2,v_2,v_2^\ast)\mbox{ for any }\alpha\in \mathfrak{gl}_{n-1}(\mathbb{C}) \end{matrix} \Big\}.$$
Note that the orthogonality condition can be replaced by $A_1\bar{A}_2-A_2\bar{A}_1+v_1\bar{v}_2^\ast-\bar{v}_2 v_1^\ast=0$.
Let $\mathcal{O}_1,\mathcal{O}_2\subset\mathcal{N}$ be any two nilpotent orbits. Set
\[U(\mathcal{O}_1,\mathcal{O}_2):=\{((A_1,v_1,v_1^\ast),(A_2,v_2,v_2^\ast))\in S|A_i\in\mathcal{O}_i\mbox{ and }(v_i,v_i^\ast)\notin (V\times\{0\})\cup (\{0\}\times V^\ast) \}. \]

\begin{prop}\label{coisotropic}
	 Let $\mathcal{O}=H_{n-1}\cdot\mathbf{e}$ with $\mathbf{e}$ regular.
	 The set $U(\mathcal{O},\mathcal{O}')$ does not contain any (non-empty) coisotropic subvariety.
\end{prop}
\begin{proof}
	It suffices to show that $R_\mathbf{e}:=U(\mathcal{O},\mathcal{O}')|_{\{\mathbf{e}\}\times V\times V^\ast}$ is not $T^\ast(V\times V^\ast)$-coisotropic. It is easy to obtain that $\mathbf{e}^{n-1}v_1=0$. Otherwise $v_1=0$. Similarly $\mathbf{e}^{n-1}v_2$ and $v_i^\ast \mathbf{e}^{n-1}$ are all zero. Thus $R_\mathbf{e}$ is not $T^\ast(V\times V^\ast)$-coisotropic due to \cite[Lemma 6.4.4]{aizenbud2013partial} and the relation $\mathbf{e}\bar{A}-A\bar{\mathbf{e}}+v_1 \bar{v}_2^\ast-\bar{v}_2 v_1^\ast=0$ for $A\in\mathcal{N}$.
\end{proof}

\begin{proof}[Proof of Proposition \ref{fourier} when $F=\mathbb{R}$ ]
	Suppose that $(x,v,v^\ast)\in Supp(f)$. Then we may assume that $x=\mathbf{e}\in Mat_{n-1,n-1}(\mathbb{R})\cap\mathcal{N}$ due to Lemma \ref{lem:guo}. Note that $SS(f)$ is $T^\ast(L_{n-1}\oplus V\oplus V^\ast)$-coisotropic. Since both $f$ and $\mathfrak{F}_{L_{n-1}}(f)$ are $H_{n-1}$-invariant, $\mathbf{e}$
	is $H_{n-1}$-distinguished (see \S2.1.3), i.e. $\mathbf{e}$ is regular nilpotent. Moreover,
	$SS(f)\subset S$. Thanks to Proposition \ref{coisotropic}, 
	\[SS(f)\subset (\mathcal{N}\times \big( (V\times\{0\})\cup (\{0\}\times V^\ast)\big))\times (\mathcal{N} \times \big((V\times\{0\})\cup (\{0\}\times V^\ast)\big)). \]
	Thus $Supp(f)\subset \mathcal{N}\times (V\times\{0\})\cup (\{0\}\times V^\ast)$. Due to \cite[Lemma 6.3.4]{aizenbud2013partial}, $f$ must be zero. This finishes the proof.
\end{proof}
\begin{rem}
	One may follow \cite[\S6]{aizenbud2013partial} to give a uniform proof of Proposition \ref{fourier} which involves more techniques and more notation when $F$ is non-archimedean.
\end{rem}

\subsection{Proof of Theorem \ref{vanishing:I}}
In this subsection, we shall give the proof of Theorem \ref{vanishing:I}.

Let $\theta$ be an involution of $G$ and $H=G^\theta$.
Then $(G,H)$ is a symmetric pair. Consider the action of $H\times H$ on $G$ by left and right translations and the action of $H$ on $L_n$ by conjugation. Let $g\in G(F)$ such that $H(F)gH(F)$ is closed in $G(F)$. Let $x=g\theta(g^{-1})$. Then
$(G_x,H_x,\theta|_{G_x})$ is called a descendant of $(G,H,\theta)$. (See \cite[\S7.2]{dima2009duke}.)
\begin{lem}\cite[Theorem 6.15]{carmeli2015stability}\label{descendant}
	Every descendant of the pair $(\GL_{2n},R_{E/F}\GL_{n})$ is a product of pairs of the form $(R_{L_1/F}\GL_{r},R_{L_2/F}\GL_r)$, $(\GL_r\times\GL_r,\triangle\GL_r)$ and  $(\GL_{2r},R_{E/F}\GL_r)$ for some $r<n$, where $L_2$ is a finite field extension over $F$ and $L_1$ is a quadratic extension of $L_2$.
\end{lem}
\begin{proof}[Proof of Theorem \ref{key:vanishing}] 
	It is enough to show that
	\[\mathscr{C}_{\mathcal{N}_n}(L_{n-1}\oplus V\oplus V^\ast)^{\tilde{H}_{n-1},\chi}=0. \]
	Pick any nilpotent orbit $\mathcal{O}$ in $\mathcal{N}_n$. 
	Thanks to Lemma \ref{lem:intertwin} and Proposition \ref{fourier}, we have
	\[\mathscr{C}_{\mathcal{O}}(L_{n-1}\oplus V\oplus V^\ast)^{H_{n-1}}=0. \]
	This finishes the proof.
\end{proof}
Finally, we give the proof of Theorem \ref{vanishing:I}.
\begin{proof}[Proof of Theorem \ref{vanishing:I}]
	Recall that $L_n=L_{n-1}\oplus V\oplus V^\ast\oplus E$ and $\tilde{H}_{n-1}$ acts on $E$ trivially.
			 We will show that 
	\begin{equation}\label{cl}
		\mathscr{C}(L_{n-1}\oplus V\oplus V^\ast)^{\tilde{H}_{n-1},\chi}=0. 
	\end{equation}
	Then \eqref{cl} implies $\mathscr{C}(L_n)^{\tilde{H}_{n-1},\chi}=0$
	by Localization Principle (see \cite[Appendix D]{dima2009duke}).

	Applying Theorem \ref{thm2.1} and Lemma \ref{descendant}, we need to consider the contribution to
	\[\mathscr{C}(L_{n-1}\oplus V\oplus V^\ast)^{\tilde{H}_{n-1},\chi} \]
	 from all descendants.
	 It is well-known that
	 \[\mathscr{C}(Mat_{r+1,r+1}(F))^{\widetilde{\GL_r(F)},\chi}=0 \]
	 where $\widetilde{\GL_r(F)}=\GL_r(F)\rtimes\langle\sigma\rangle$ and $\sigma$ acts on $\GL_r(F)$ by
	 \[\sigma\cdot g=(g^t)^{-1} \]
	 for $g\in\GL_r(F)$. (See \cite{AG2010annals} for the non-archimedean case and \cite{AG2009selecta,sunzhu2012annals} for the archimedean case.) Thus the descendants  $(R_{L_1/F}\GL_r,R_{L_2/F}\GL_r)$ and $(\GL_r\times\GL_r,\triangle\GL_r)$ do not contribute to $\mathscr{C}(L_n)^{\tilde{H}_{n-1},\chi}$.
	Therefore it suffices to show that
	\[\mathscr{C}(R(L_{r-1}\oplus V\oplus V^\ast))^{\tilde{H}_{r-1},\chi}=0\Longrightarrow\mathscr{C}(Q(L_{r-1}\oplus V\oplus V^\ast))^{\tilde{H}_{r-1},\chi}=0 \]
	for all $r$.
	 Thus 
	$$\mathscr{C}_{\mathcal{N}_r}(L_{r-1}\oplus V\oplus V^\ast)^{\tilde{H}_{r-1},\chi}=0$$
	(see Theorem \ref{key:vanishing})  implies $\mathscr{C}(L_{n-1}\oplus V\oplus V^\ast)^{\tilde{H}_{n-1},\chi}=0$.
	This finishes the proof.	
\end{proof}

\section{Proof of Theorem \ref{thmA}}\label{sect:proofA}
Following \cite{offen2011,kemarsky}, we will give the proof of Theorem \ref{thmA} in this section.
\subsection{Proof of Theorem \ref{thm:group}}
This subsection focuses on the proof of Theorem \ref{thm:group}.
Define
\[\mathcal{H}_n:=\GL_{n}(E)\times\GL_{n}(E) \]
and $\tilde{\mathcal{H}}_n=\mathcal{H}_n\rtimes\langle\sigma\rangle$, where the action is given by
\[\sigma(g_1,g_2)=((\bar{g}_2^{-1})^t,(\bar{g}_1^{-1})^t) \]
for $g_i\in\GL_n(E)$.
Let $\tilde{\mathcal{H}}_n$ act on $\GL_{2n}(F)$ by
\[(g_1,g_2)\cdot x=g_1xg_2^{-1} \]
and  $\sigma\cdot x=\bar{x}^t$
for $g_i\in\GL_{n}(E)$ and $x\in \GL_{2n}(F)$, which induces an action of $\tilde{\mathcal{H}}_n$ on $\mathscr{C}(\GL_{2n}(F))$. Let $\mathcal{H}_{n,x}$ (resp. $\tilde{\mathcal{H}}_{n,x}$) denote the stabilizer of $x$ in $\mathcal{H}_n$ (resp. $\tilde{\mathcal{H}}_n$).

	\begin{lem}\label{Jac:lem} \cite[Proposition 1.2]{guo1996cjm}
		The double cosets $H_nxH_n$, where
		$$x\theta(x^{-1})=\begin{pmatrix}
		A&0&0&B_A&0&0\\0&-\mathbf{1}_{p}&0&0&0&0\\ 0&0&\mathbf{1}_q&0&0&0\\
		\bar{B}_A&0&0&A&0&0\\
		0&0&0&0&-\mathbf{1}_p&0\\
		0&0&0&0&0&\mathbf{1}_q
		\end{pmatrix},$$
		exhaust all closed orbits
		in $\GL_{2n}(F)$, where $A\in Mat_{\nu,\nu}(F)$ is semisimple without eigenvalues $\pm1$,
		$\nu+p+q=n$,  $\mathbf{1}_p,\mathbf{1}_q,\mathbf{1}_\nu$ are identity matrices and $B_A\in Mat_{\nu,\nu}(E)$ satisfies
		$A^2-\mathbf{1}_{\nu}=B_A\bar{B}_A$  and $AB_A=B_A A$. 
\end{lem}
\begin{lem}\label{lem4.2}
	One has $\mathscr{C}(\GL_{2n}(F))^{P_E\times\GL_n(E)}=\mathscr{C}(\GL_{2n}(F))^{\GL_n(E)\times\GL_n(E) }$.
\end{lem}
\begin{proof} Note that $\mathscr{C}(\GL_{2n}(F))^{K\times \GL_{n}(E)}\cong \mathscr{C}(\GL_{2n}(F)/\GL_{n}(E))^K$ for any subgroup $K$ of $\GL_{2n}(F)$ (see \cite[Lemma 3.7]{kemarsky}). Thus it  suffices to show that any element in $\mathscr{C}(\GL_{2n}(F)/\GL_{n}(E))^{P_E}$
	is invariant under transposition. Indeed, we shall prove that
	\[\mathscr{C}(\GL_{2n}(F)/\GL_{n}(E))^{\tilde{H}_{n-1},\chi}=0. \]
	Applying Theorem \ref{thm:duke}, it is enough to show that
	\begin{equation}\label{normal} \mathscr{C}(N_{H_nxH_n,x}^{\GL_{2n}(F)})^{\tilde{\mathcal{H}}_{n,x}\cap\tilde{H}_{n-1},\chi}=0
\end{equation}
	for any closed orbit $H_nxH_n$ in $\GL_{2n}(F)$. Note that if $x\in H_n$, then 
	$N_{H_nxH_n,x}^{\GL_{2n}(F)}\cong L_n$ and $\mathcal{H}_{n,x}\cong \GL_{n}(E),\mathcal{H}_{n,x}\cap (H_{n-1}\times\GL_{n}(E))\cong H_{n-1}$. Then \eqref{normal} follows from Theorem \ref{vanishing:I}. According to Lemma \ref{Jac:lem}, we separate the proof into two cases.
	\begin{enumerate}
		\item 
	If $x=\begin{pmatrix}
&&	\mathbf{1}_p\\&\mathbf{1}_q\\\mathbf{1}_p\\ &&&\mathbf{1}_q
	\end{pmatrix}$ with $p+q=n$, then 
	$$N_{H_nxH_n,x}^{\GL_{2n}(F)}\cong \frac{\mathfrak{gl}_{2n}(F)}{\mathfrak{gl}_n(E)+ Ad_x\mathfrak{gl}_n(E)}\cong Mat_{p,p}(E)\oplus L_q,$$ $\mathcal{H}_{n,x}\cong \GL_p(E)\times\GL_q(E)$ and
	\[\mathcal{H}_{n,x}\cap (H_{n-1}\times\GL_{n}(E))\cong \GL_p(E)\times H_{q-1}. \]
	The action of $H_{n,x}$ on $Mat_{p,p}(E)\oplus Mat_{q,q}(E)$ is given by
	\[(g_1,g_2)\cdot (x,y)=(\bar{g}_1xg_1^{-1},g_2y\bar{g}_2^{-1}) \]
	for $g_1\in\GL_p(E),g_2\in\GL_q(E)$, $x\in Mat_{p,p}(E)$ and $y\in Mat_{q,q}(E)$. Thus \eqref{normal} follows from Theorem \ref{vanishing:I}. 
	\item If $x$ satisfies 
	\[x\theta(x^{-1})=\begin{pmatrix}
	A&&B_A\\ &\mathbf{1}_{n-\nu}\\ \bar{B}_A&&A\\&&&\mathbf{1}_{n-\nu}
	\end{pmatrix}, \]
	then we may assume that $A$ is a scalar and $A^2\neq\mathbf{1}_\nu$. It is easy to see that
	$N_{H_nxH_n,x}^{\GL_{2n}(F)}\cong Mat_{\nu,\nu}(F)\oplus Mat_{n-\nu,n-\nu}(E),\mathcal{H}_{n,x}\cong \GL_\nu(F)\times \GL_{n-\nu}(E)$ and
	\[\mathcal{H}_{n,x}\cap(H_{n-1}\times\GL_{n}(E))\cong \GL_\nu(F)\times H_{n-1-\nu} \]
	The action of $H_{n,x}$ on $Mat_{\nu,\nu}(F)\oplus Mat_{n-\nu,n-\nu}(E)$ is given by
	\[(g_1,g_2)\cdot (x,y)=(g_1xg_1^{-1},g_2y\bar{g}_2^{-1}) \]
	for $g_1\in\GL_\nu(F),g_2\in\GL_{n-\nu}(E),x\in Mat_{\nu,\nu}(F)$ and $y\in Mat_{n-\nu,n-\nu}(E)$.
	In a similar way, \eqref{normal} holds.
\end{enumerate}
This finishes the proof.
\end{proof}
\begin{proof}[Proof of Theorem \ref{thm:group}] Recall that $X_n=\GL_{2n}(F)/\GL_n(E)$.
	From the proof of Lemma \ref{lem4.2}, we obtain that
	\begin{equation} \label{equal:C}
	\mathscr{C}(X_n)^{\tilde{H}_{n-1},\chi}=0 \end{equation}
		From a general principle of \say{distribution versus Schwartz distribution} (see \cite[Theorem 4.0.2]{dima2009duke}), the equality \eqref{equal:C} implies
		\begin{equation}
		\label{equal:D}
		\mathscr{D}(X_n)^{\tilde{H}_{n-1},\chi}=0.
		\end{equation}
			Note that $H_{n-1}\subset P_E$ and that the mirabolic subgroup $P_E$ and its transpose $P_E^t$ generate $\GL_n(E)$.
		Thus one has $\mathscr{D}(X_n)^{P_E}=\mathscr{D}(X_n)^{\GL_{n}(E)}$. This finishes the proof.
\end{proof}
\subsection{Proof of Theorem \ref{thmA}}\label{subsect:A}
This subsection focuses on the proof of Theorem \ref{thmA}. Let us recall the following lemma appearing in \cite{offen2011,kemarsky}.
\begin{lem}\cite[Corollary 3.5]{kemarsky}\label{ker:cor}
	Let $\pi$ be an irreducible smooth admissible representation of $\GL_{2n}(F)$. Let
	$\pi^\vee$ (resp. $\pi^\ast$) be the contragredient (reps. linear dual) of $\pi$. Then there exists an injective morphism from
	$\pi^\ast\otimes(\pi^\vee)^\ast$  to the space  $\mathscr{D}(\GL_{2n}(F))$ consisting of all distributions on $\GL_{2n}(F)$ as $\GL_{2n}(F)\times\GL_{2n}(F)$-modules, denoted
	one of those injective morphisms by $A_\pi$. 
\end{lem}

Now we are ready to give a proof of Theorem \ref{thmA}.
\begin{proof}[Proof of Theorem \ref{thmA}]
	Note that $$\mathscr{D}(X_n)^K\cong\mathscr{D}(\GL_{2n}(F)/\GL_{n}(E))^K\cong\mathscr{D}(\GL_{2n}(F))^{K\times\GL_n(E) }$$ for any subgroup $K$ of $\GL_{n}(E)$. Thus
Theorem \ref{thm:group} implies that
	\begin{equation}\label{equ:dist}
	\mathscr{D}(\GL_{2n}(F))^{P_E\times\GL_{n}(E) }=\mathscr{D}(\GL_{2n}(F))^{\GL_{n}(E)\times\GL_n(E) }. \end{equation}
	Let $\pi$ be a $\GL_n(E)$-distinguished representation of $\GL_{2n}(F)$. Then its contragredient representation $\pi^\vee$ is also $\GL_{n}(E)$-distinguished. Denote by $\pi^\ast$ the linear dual of $\pi$. Take two non-zero linear forms  $\mu\in(\pi^\ast)^{P_E}$ and
	 $\lambda\in ((\pi^\vee)^\ast)^{\GL_{n}(E)}$ . Then Lemma \ref{ker:cor} implies
	\[0\neq A_\pi(\mu\otimes\lambda)\in\mathscr{D}(\GL_{2n}(F))^{P_E\times\GL_{n}(E)} \]
which	is $\GL_n(E)\times\GL_n(E)$-invariant as well by the identity \eqref{equ:dist}. Since $A_\pi$ is injective due to Lemma \ref{ker:cor}, $\mu\otimes\lambda\in(\pi^\ast\otimes(\pi^\vee)^\ast )^{\GL_{n}(E)\times\GL_{n}(E) }$. Therefore $\mu\in(\pi^\ast)^{\GL_{n}(E)}$. This finishes the proof of Theorem \ref{thmA}.
\end{proof}

\subsection*{Acknowlegement} The author would like to thank Dmitry Gourevitch for his careful reading and useful comments for the first version. 
This work was partially supported by
the ERC, StG grant number 637912 and ISF grant 249/17.

\bibliographystyle{amsalpha}
\bibliography{mira}

\end{document}